\documentclass{amsart}

\usepackage{amsmath, amsthm, amssymb}
\usepackage{epsfig}
\usepackage{array}
\usepackage{amsmath}
\usepackage{amsfonts}
\usepackage{amsfonts}
\usepackage{amsmath}%
\usepackage{amsfonts}%
\usepackage{amssymb}%
\usepackage{graphicx}%------------------------------------------------------------

\usepackage{amsfonts}
\usepackage{amssymb}
\usepackage{amsmath,amsxtra,amssymb}

\numberwithin{equation}{section}

\newtheorem{theorem}{Theorem}%[section]
\newtheorem{lemma}[theorem]{Lemma}

\newcommand{\field}[1]{\mathbb{#1}}

\newcommand{\id}{{\iota}}
\newcommand{\ot}{{\otimes}}
\newcommand{\om}{{\omega}}

\newcommand{\LL}{{\mathcal{L}^{\infty}(\G)}}

\newcommand{\PG}{{\mathcal{P}(\G)}}

\newcommand{\CZ}{{\mathcal{C}_0(\G)}}

\newcommand{\G}{\mathbb G}

\title %[Limit Theorems]
{A Limit Theorem for Discrete Quantum Groups}
\author{Mehrdad Kalantar}
\address{ School of Mathematics and Statistics,
Carleton University, Ottawa, ON, Canada}
\email{mkalanta@math.carleton.ca}


\begin{document}
%\maketitle
%============================
%  ABSTRACT
%============================
\begin{abstract}
We consider the {\it concentration functions problem}
for discrete quantum groups;
we prove that if $\G$ is a discrete quantum group,
and $\mu$ is an irreducible state in $l^1(\G)$,
then the convolution powers $\mu^n$,
considered as completely positive maps on $c_0(\G)$,
converge to zero in strong operator topology.
\end{abstract}

%\begin{document}

\maketitle

Consider an irreducible probability measure $\mu$ on a discrete group $G$
(that is, $G$ is generated as a semi-group, by the support of $\mu$);
then under what conditions on $G$ or $\mu$, does the sequence
\begin{equation}\label{1123}
f_n(K) = \sup\,\{\mu^n (Kx^{-1}) \,:\, x\in G\, \}
\end{equation}
converges to zero for every finite set $K \subset G$\,?
This problem is known as the {\it concentration functions problem}.
In \cite{HM}, Hofmann--Mukherjea considered, and partially answered this problem
for the case of locally compact groups. In particular they proved that the sequence (\ref{1123})
converges to zero for all discrete $G$, finite $K\subseteq G$, and irreducible $\mu\in l^1(G)$.
It is not hard to see that the convergence of the concentration functions
to zero is equivalent to the convergence of the convolution operators $\mu^n$ on $c_0(\G)$, to zero
in the strong operator topology.

Here we prove a non-commutative version of this well-known
classical result for the case of discrete quantum groups (Theorem \ref{HM}).

%We note that the corresponding proofs in the classical setting
%are based on some deep structural results about locally compact groups,
%which by no means are available in the quantum case.
%Therefore, here we have to provide new and purely functional analytic arguments.

%%%%%%%%%%%%%%%%%%%%%%%%%%%%%%%%%%%%%%%%%%%%%
%%%%%%%%%%%%%%%%%%%%%%%%%%%%%%%%%%%%%%%%%%%%%

%First, let us briefly introduce the terminology that we will be using in this paper.
%For more details on locally compact quantum groups we refer the reader to \cite{KV1}.

We recall that a {\it discrete quantum group} $\G$ (in the sense of \cite{KV1}) is a
quadruple $( l^\infty(\G), \Gamma, \varphi, \psi)$, where
$\Gamma:  l^\infty(\G)\to  l^\infty(\G) \bar\otimes  l^\infty(\G)$
is a co-associative co-multiplication
on the von Neumann algebra $ l^\infty(\G)$,
and $\varphi$ and  $\psi$ are  (normal faithful semi-finite) left and right
Haar weights on $ l^\infty(\G)$, respectively.

The preadjoint of the co-multiplication $\Gamma$ induces
an associative multiplication on $ l^1(\G):=  l^\infty(\G)_*$ given by
the convolution
\[
%\star :  l^1(\G)\tp  l^1(\G)\ni \mu\ot \nu \, \longmapsto\,
\mu \star \nu\, =\, (\mu\otimes \nu)\,\Gamma\,.
% = \nu (\mu \otimes \id)\Gamma
%\in  l^1(\G)
\]
%such that  $ l^1(\G)$ contains $ l^1(\G):=  l^\infty(\G)_*$ as a norm closed two-sided ideal.
%Therefore, for each $\mu\in l^1(\G)$, we obtain a pair of completely bounded maps
%\begin{equation*}
%f\,  \longmapsto \, \mu \star f\ \  ~ \mbox{and } ~ \ \ f \, \longmapsto\, f \star \mu 
%\end{equation*}
%on $ l^1(\G)$ through the left and right convolution products of $ l^1(\G)$.
With this product $l^1(\G)$ is a completely contractive unital Banach algebra.
The convolution actions $x\star\mu:=(\id\otimes\mu)\Gamma(x)$ and $\mu\star x:=(\mu\otimes\id)\Gamma(x)$
are normal completely bounded maps on $ l^\infty(\G)$.

The \emph{reduced quantum group $C^*$-algebra} associated to $\G$ is denoted by $ c_0(\G)$, which
is a weak$^*$ dense $C^*$-subalgebra of $ l^\infty(\G)$.
Similar to the classical case, we have the duality $ l^1(\G) = c_0(\G)^{*}$.

%%%%%%%%%%%%%%%%%%%%%%%%%%%%
%%%%%%%%%%%%%%%%%%%%%%%%%%%%

We denote by $\PG$ the set of all `{\it quantum probability measures}' on $\G$
(i.e. all normal states in $ l^1(\G)$).
For any such element the convolution action is a \emph{Markov operator}, i.e.
a unital normal completely positive map, on $ l^\infty(\G)$.

%In the classical setting, when studying random walks and associated
%boundaries on a locally compact group $G$,
%in order to rule out trivialities,
%one usually works with measures whose support generates $G$ as a closed
%semigroup (or group).
%Similarly, in the quantum setting we
%restrict ourselves to those {quantum measures} possessing such property;
A state $\mu\in \PG$ is \emph{irreducible}
if for every non-zero element $x\in  l^\infty(\G)^+$
there exists $n\in\field{N}$ such that
$\langle\, x\, ,\, \mu^n\,\rangle\, \neq\, 0$.

%If $\mu\in \PG$, then there exists a unique strictly continuous
%state extension of $\mu$ to a state on $\CG$, the multiplier $C^*$-algebra
%of $ c_0(\G)$, which we still denote by $\mu$.

%%%%%%%%%%%%%%%%%%%%%%%%%%%%%%%%%%%%%%%%%%%%
%%%%%%%%%%%%%%%%%%%%%%%%%%%%%%%%%%%%%%%%%%%%

% % % % % % % % % % % % % % % %

%\begin{lemma}\label{nondeg}
%Let $\mu\in \PG$ be irreducible.
%Then for every non-zero $x\in \CG^+$ there
%exists $n\in\field{N}$ such that
%$\langle\, x\, ,\, \mu^n\,\rangle\, \neq\, 0$.
%\end{lemma}
%\begin{proof}
%Let $x\in \CG^+$ be non-zero, and let $a\in c_0(\G)^+$ be such that
%$\|a\| = 1$ and $a^{\frac 12}x^{\frac 12}\neq 0$. Then
%we have $$ c_0(\G)^+\ni x^{\frac 12}a x^{\frac 12}\leq x.$$ Now since $\mu$ is irreducible,
%there exists $n\in\field{N}$ such that
%$$0\, <\, \langle\, x^{\frac 12}a x^{\frac 12}\, ,\, \mu^n\,\rangle\,\leq\,\langle\, x \,,\, \mu^n\,\rangle\,.$$
%\end{proof}

% % % % % % % % % % % % % %

%In applications (both in the classical and the quantum setting), one usually
%deals with those (quantum) probability measures $\mu$ which have the property that
%$\mu^m$ is not singular with respect to the Haar measure, for some $m$.
%We prove our main result for this class of quantum measures;
%a state $\mu\in\PG$ is said to be spread-out if there is $m\in\mathbb N$
%such that $\mu^m$ can be decomposed as
%\[
%\mu^m\,=\, f\,+\,\nu\,,
%\]
%where $0\neq f\in l^1(\G)^+$ and $\nu\in l^1(\G)^+$.

% % % % % % % % % % % % % % % % % % % % %

\begin{lemma}\label{0-2}
Let $\G$ be a discrete quantum group, and let $\mu\in\PG$ be irreducible.
Then there exists $k\in\mathbb{N}$ such that
\[
\lim_n \|\,\mu^{n+k} \, -\, \mu^n\,\| \, =\,0\,.
\]
\end{lemma}
\begin{proof}
Since $\mu$ is irreducible, it follows from the definition that
$ \mu_0:=\sum_{n=1}^\infty\,2^{-n} \mu^n$
is a faithful normal state.
Now, let $\psi$ be the right Haar weight of $\G$, then we have
\[
\psi (x\star \mu)\,1 \,=\,  (\psi\otimes \id)\,\Gamma (x\star\mu)
%= (\psi\otimes \Phi_{\mu}) \Gamma(x)
\,=\, \left((\psi\otimes \id)\, \Gamma(x)\right)\star \mu
\,=\, \psi(x)\,1
\]
for all $x \in l^\infty(\G)^{+}$. Since $\psi$ is faithful,
it follows that the map $x \mapsto x\star \mu$ is faithful on $l^\infty(\G)$.
This, in particular, implies that $\mu_0\star\mu^m$ is faithful on $l^\infty(\G)$
for all $m\in\mathbb N$. Thus, for every positive $x\in l^\infty(\G)$
and $n\in\mathbb N$, there are $n \leq m_1, m_2$ such that
$\langle\,\mu^{m_1},x\,\rangle\,\neq\,0$ and $\langle\,\mu^{m_2},x\,\rangle\,\neq\,0$.
Let $e\in l^\infty(\G)$ be the central minimal projection obtained from the trivial representation
of the dual compact quantum group, and suppose that $n, k \in \mathbb N$ are such that
\[
\langle\,\mu^n\,,\,e\,\rangle\,\neq\,0\ \ \ \ \ \text{and}\ \ \ \ \ \langle\,\mu^{n+k}\,,\,e\,\rangle\,\neq\,0\,.
\]
Then we have $\lambda\delta_e \leq \mu^n$ and $\lambda\delta_e \leq \mu^{n+k}$
for some positive $\lambda\in\mathbb R$,
where $\delta_e = e\psi$.
%
%$0\neq f\star e\in \CG^+$ (\cite[Theorem 2.4]{Volker}), and so $\mu$ being irreducible
%implies by Lemma \ref{nondeg} that
%\begin{equation}\label{00}
%\langle\,e\,,\,\mu^k\,\star\,f\,\rangle\,=\, \langle\,f\,\star\,e\,,\,\mu^k\,\rangle\,\neq\,0
%\end{equation}
%for some $k\in\mathbb N$.
%So, if $e'$ denotes the support projection of $\mu^k\star f\in l^1(\G)^+$, it follows
%from (\ref{00}) that $e\,e'\neq 0$, whence
%\[
%\|\,f\,-\,\mu^k\star f\,\|\,<\,\|f\|\,+\,\|\mu^k\star f\|\,,
%\]
%by \cite[Theorem III.4.2]{Tak1}.
%Then, since we have
%$$\mu^{m+j}\, =\, \mu^j\star f\, +\, \mu^j\star \nu\,,$$
%we obtain
%So, we get
%\begin{eqnarray*}
%\|\,\mu^m\,-\,\mu^{m+k}\,\|&\leq&\|\,f\,-\,\mu^k\star f\,\|\,+\|\,\nu\,-\,\mu^k\star \nu\,\|\\
%&<&\|f\|\,+\,\|\mu^k\star f\|\,+\,\|\,\nu\,-\,\mu^k\star \nu\,\|\\
%&\leq& \|\,\mu^m\,\|\,+\,\|\,\mu^{m+k}\,\|\,.
%\end{eqnarray*}
%Then a straightforward calculation
%yields that $\om\leq \mu^m$ and $\om\leq \mu^{k+m}$, where
%\[
%\om\,:=\, \frac 12 \big(\mu^m\,+\,\mu^{k+m}\, -\, |\,\mu^m\, -\, \mu^{k+m}\,|\big)\in l^1(\G)^+.
%\]
Hence, lemma follows from the non-commutative 0-2 law \cite[Proposition 2.12]{NeshTuset}.
\end{proof}

\begin{theorem}\label{HM}
%Let $\G$ be a (infinite dimensional) discrete quantum group, and let $\mu\in\PG$ be an irreducible
%quantum probability measure. Then
Let $\G$ be a (infinite dimensional) discrete quantum group, and let $\mu\in\PG$ be irreducible. Then
\begin{equation}\label{l2}
\|\,x\star \mu^n\,\|\, {\longrightarrow}\,0
\end{equation}
for all $x\in c_0(\G)$.
\end{theorem}
\begin{proof}
We first show that there exists a subnet $(\mu^{n_i})_i$ such that
\begin{equation}\label{1-k}
\lim_{i}\,\mu^{n_i+s}\, =\,0
\end{equation}
in the $\sigma( l^1(\G)\,,\, c_0(\G))-$topology, for all $s\in\mathbb N$.
So, let $\mathcal{F}$ be a Banach limit, and define
\[
\nu \,=\, \text{weak*-}\lim_{\mathcal{F}}\,\mu^n\,\in l^1(\G)^+.
\]
Then $\nu\star\mu = \mu\star\nu = \nu$, and so
\[
\mu\star(\nu\star x)\, = \,(\mu\star\nu)\star x \,=\, \nu\star x
\]
for all $x\in c_0(\G)$.

Now, let $x\in c_0(\G)^+$, and suppose that $\eta\in\PG$ is such that
\[\langle\, \nu\star x\,,\, \eta\,\rangle \,=\, \|\nu\star x\|\,.\]
Then we get
\begin{eqnarray*}
\langle\, (\,\|\nu\star x\|\,1\, -\, \nu\star x\,)\,\star\eta\, ,\, \mu^n\,\rangle &=&
\langle\, \mu^n\star(\,\|\nu\star x\|\,1\, -\, \nu\star x\,)\, ,\, \eta\,\rangle\\ &=&
\langle\, \|\nu\star x\|\,1\, -\, \nu\star x \,,\, \eta\,\rangle\\ &=&
\|\nu\star x\|\,-\|\nu\star x\|
\, =\,0
\end{eqnarray*}
for all $n\in\mathbb N$.
Hence, it follows from the irreducibility of $\mu$, that
\[(\id\ot\eta)\,\Gamma(\,\|\nu\star x\|\,1\, -\, \nu\star x\,)\,=\,
(\,\| \nu \star x\|\,1\, -\, \nu\star x\,)\,\star\eta\,=\,
0\,,\]
and therefore, we obtain
\[
0 \,=\, \psi \big((\id\ot\eta)\,\Gamma(\,\|\nu\star x\|\,1\, -\, \nu\star x\,)\big) \,= \,
\psi \big(\,\|\nu\star x\|\,1\, -\, \nu\star x\,\big)\,,
\]
where $\psi$ is the right Haar weight.
Since $\|\nu\star x\|\,1\, -\, \nu\star x\,\geq\, 0$ and $\psi$ is faithful, it follows that
\[
\|\,\nu\star x\,\|\,1\, -\, \nu\star x\, =\, 0\,.
\]
But since $1\notin  c_0(\G)$, this shows that
$\nu\star x = 0$ for all $x\in c_0(\G)$. From this, we see that $\nu$ vanishes on
\[
\{\,(\id\ot\om)\,\Gamma(x)\,:\, \om\in c_0(\G)^*,\, x\in c_0(\G)\,\},
\]
and since the latter is dense in $ c_0(\G)$, we conclude that $\nu = 0$.
Hence, for each $0\leq x\in c_0(\G)$ we have
\begin{equation}\label{m1}
0\, \leq \,\liminf_n \,\langle\,x\,,\,\mu^n\,\rangle \,\leq\,\mathcal{F}\left(\langle\,x\,,\,\mu^n\,\rangle_{n=1}^{\infty}\right)\,=\,0\,.
\end{equation}
Now, let $x_1$, $x_2$, $\dots$, $x_n$ be positive elements in $ c_0(\G)$, and define
\[
x_0\,=\,\sum_{k=1}^n\,{x_k}\,\in\, c_0(\G)^+.
\]
Then, by (\ref{m1}), we can choose a subnet $(\mu^{n_j})$ of $(\mu^{n})$ such that
$\langle\,x_0,\mu^{n_j}\,\rangle\rightarrow 0$.
%\[
%\lim_j\,\langle\,x_0\,,\,\mu^{n_j}\,\rangle\, = \, 0.
%\]
Thus, it follows that
\[
\lim_j\,\langle\,x_k\,,\,\mu^{n_j}\,\rangle\, = \, 0.
\]
for all $1\leq k \leq n$.
This shows that we can find a subnet $(\mu^{n_i})$ of $(\mu^{n})$ such that
$
\mu^{n_i} \, {\longrightarrow}\,0
$
%converges to zero in the weak* topology.
in the $\sigma( l^1(\G)\,,\, c_0(\G))-$topology, whence (\ref{1-k}) follows.
%Then it also follows that
%\begin{equation}\label{1-k}
%\mu^{n_i+s}\, {\longrightarrow}\,0
%\end{equation}
%in the $\sigma( l^1(\G)\,,\, c_0(\G))-$topology for all $s\in\mathbb N$.

Now, towards a contradiction, suppose that there exists $0\leq x\in c_0(\G)$ such that
the limit in (\ref{l2}) does not converge to zero.
Since the sequence $\|\,x\star \mu^n\,\|$
is positive and non-increasing, it has a limit. So, there exists $\alpha > 0$
such that
$
\|\,x\star \mu^n\,\|\,\geq \alpha
$
for all $n\in\mathbb{N}$, and therefore we can find $\om_n\in\PG$ such that
\begin{equation}\label{111}
\langle\,x\star \mu^n\,,\,\om_n\,\rangle\,\geq \alpha
\end{equation}
for all $n\in\mathbb{N}$. Now, assume that
\[
\rho\, =\, \lim_j\, \om_{n_{i_j}}\,\star\,\mu^{n_{i_j}}
\]
is a weak* cluster point of
$\{\om_{n_i}\star\mu^{n_i}\}$ in the unit ball of $ l^1(\G)$. Then (\ref{111}) implies that
%\begin{equation}\label{112}
$\langle x\,,\,\rho\,\rangle\geq \alpha\,$.
%\end{equation}
Moreover, by Lemma \ref{0-2}, there exists $k\in\mathbb{N}$ such that
\begin{eqnarray*}
\langle\,a\,,\,\rho\star\mu^k\,-\,\rho\,\rangle &=&
\lim_j\,\langle\,a\,,\,\om_{n_{i_j}}\star\mu^{n_{i_j}}\star\mu^{k}\,-\,\om_{n_{i_j}}\star\mu^{n_{i_j}}\,\rangle\\
&\leq&
\lim_j\, \|\,a\,\|\,\|\,\mu^{k+n_{i_j}}\,-\,\mu^{n_{i_j}}\,\|\,=\,0
\end{eqnarray*}
for all $a\in c_0(\G)$. Hence,
\[
\rho\,=\,\rho\star\mu^{k}\,=\,\rho\star\mu^{nk}
\]
for all $n\in\mathbb N$, which in particular, implies
%\begin{eqnarray*}
%\langle\,x\,,\,\rho\star\mu^k\,-\,\rho\,\rangle &=&
%\lim_j\,\langle\,x\,,\,\om_{n_{i_j}}\star\mu^{n_{i_j}}\star\mu^{k}\,-\,\om_{n_{i_j}}\star\mu^{n_{i_j}}\,\rangle\\
%&\leq&
%\|\,x\,\|\,\|\,\mu^{k+{n_{i_j}}}\,-\,\mu^{n_{i_j}}\,\|\,\longrightarrow\,0\,.
%\end{eqnarray*}
%Hence, we get
\begin{equation}\label{113}
\alpha\,\leq\,\langle\,x\,,\,\rho\,\rangle\, = %\,\langle\,x\,,\,\rho\star\mu^{k}\, \rangle\,=
\,\langle\,x\,,\,\rho\star\mu^{nk}\,\rangle
\end{equation}
for all $n\in\mathbb N$.
But, on the other hand, from (\ref{1-k}) it follows that there exists $m\in\mathbb N$ such that
\[
\langle\,x\,,\,\rho\star\mu^{m+s}\,\rangle\,=\,\langle\,x\star\rho\,,\,\mu^{m+s}\,\rangle\, <\,\frac{\alpha}{2}
\]
for all $1\leq s\leq k$. Since $k$ divides one of $m+1$, $m+2$, \dots, $m+k$, this
contradicts (\ref{113}), and therefore finishes the proof.
%Therefore, we obtain
%\[
%\langle\,x\star\rho\,,\,\mu^{m+s}\,\rangle\, <\,\frac{\alpha}{2}
%\]
% which, by Theorem \ref{DM}, gives
%\[
%\langle\,a\,,\,\rho\,\rangle\,=\,\langle\,a\,,\,\rho\star\mu^{nk}\,\rangle\,=\,
%\langle\,a\,\star\, \rho\,,\,\mu^{nk}\,\rangle\,\longrightarrow\,0
%\]
%for all $a\in c_0(\G)$, whence $\rho = 0$. This contradiction finishes the proof.
\end{proof}

\noindent
{\bf Remark 1.}\
The proof of Lemma \ref{0-2} can be modified to prove
the statement for the case of a locally compact quantum group
$\G$ whose von Neumann algebra $\LL$ contains a non-zero central abelian
projection, and $\mu$ is a state on $\CZ$ with at least one non-singular convolution power.
Then the same proof yields Theorem \ref{HM} in this case; in particular,
this provides a new proof for the concentration functions problems in the
classical setting, for spread-out probability measures on locally compact groups.\\

%\begin{remark}
\noindent
{\bf Remark 2.}\
The convergence of convolution powers $\mu^n$ to zero in the weak* topology,
in the classical case, holds for a more general class of probability measures $\mu$,
namely those whose support generates $G$ as a group, rather than a semigroup.
But this is not the case for Theorem \ref{HM}. For a counter-example, consider the
additive group of integers $(\mathbb Z\,,\,+)$, and the probability measure $\delta_1$,
the Dirac measure at $1$. Then the group generated by the support of $\delta_1$ is $\mathbb Z$,
but every convolution power $\delta_1^n = \delta_n$ induces an isometry on $c_0(\mathbb Z)$.
%\end{remark}

\bibliographystyle{plain}

\begin{thebibliography}{99}






%\bibitem {Derr} Y. Derriennic,
%\textit{Lois ``z\'{e}ro ou deux'' pour les processus de Markov. Applications aux marches al\'{e}atoires.},
%Ann. Inst. H. Poincar\'{e} Sect. B (N.S.) 12 (1976), no. 2, 111--129. 


\bibitem {HM} K. H. Hofmann \and A. Mukherjea,
\textit{Concentration functions and a class of noncompact groups.},
Math. Ann. \textbf{256} (1981), no. 4, 535--548.




\bibitem {I} 
M. Izumi, \textit{Non-commutative Poisson boundaries and compact quantum group actions}, 
Adv. Math. \textbf{169} (2002), no. 1, 1--57. 





\bibitem {KV1} J. Kustermans \and S. Vaes, \textit{Locally compact quantum groups},
Ann. Sci. Ecole Norm. Sup. \textbf{33} (2000), 837--934.


%\bibitem {Mukh} A. Mukherjea,
%\textit{Limit theorems for probability measures on non-compact groups and semi-groups},
%Z. Wahrscheinlichkeitstheorie und Verw. Gebiete \textbf{33} (1976), no. 4, 273--284.





\bibitem{NeshTuset}  S. Neshveyev \and L. Tuset,
\textit{The Martin boundary of a discrete quantum group},
J. Reine Angew. Math. \textbf{568} (2004), 23--70.






\end{thebibliography}

\end{document}